\documentclass[11pt,amstex,amssymb]{amsart}
\usepackage{amsmath,amsthm,amsfonts,amssymb,amscd}
\usepackage[latin1]{inputenc}
\usepackage[all]{xy}
\usepackage[dvips]{graphicx}
\usepackage{amsmath}
\usepackage{amsthm}
\usepackage{amsfonts}
\usepackage{amssymb}%
\usepackage{amsmath}
\usepackage{amssymb}
\usepackage{amsthm}
\def\res{\rm{{res}}}
\newcommand\Inv{\rm{Inv}}
\def\fa{{\mathcal{F}}}
\def\O{{\mathcal{O}}}

\def\M{{\mathcal{M}}}
\def\D{{\mathcal{D}}}
\def\U{{\mathcal{U}}}
\def\B{{\mathcal{B}}}

\title[Stable singularities of  holomorphic vector fields]{Stable
singularities of  holomorphic vector fields}
\author{V. Le\'on}
\author{B. Sc\'ardua}
\address{V. Le\'on. ILACVN - CICN, Universidade Federal da Integração Latino-Americano, Parque tecnológico de Itaipu, Foz do Iguaçu-PR, 85867-970 - Brazil}
\email{victor.leon@unila.edu.br}
\address{B. Sc\'ardua. Instituto de Matem\'atica - Universidade Federal do Rio de Janeiro,
CP. 68530-Rio de Janeiro-RJ, 21945-970 - Brazil}

\email{scardua@im.ufrj.br}


\date{}

\begin{document}

\begin{abstract}
 We consider germs of  holomorphic vector fields with an isolated singularity  at the origin $0\in\mathbb{C}^2$. We introduce a notion of stability, similar to  {\it  Lyapunov} stability. For such  a germ, called  {\it $L$-stable}  singularity, either the corresponding foliation admits a holomorphic first integral,  or it is a real logarithmic foliation singularity. A notion of $L$-stability is also naturally introduced for a leaf of a foliation. In the complex codimension one case, for holomorphic foliations, the holonomy groups of $L$-stable leaves are proved to be abelian, of a suitable type. This implies the existence of local closed meromorphic one-forms defining the foliation, in a neighborhood of $L$-stable leaves. \end{abstract}

\maketitle

\newtheorem{Theorem}{Theorem}[section]
\newtheorem{Corollary}{Corollary}[section]
\newtheorem{Proposition}{Proposition}[section]
\newtheorem{Lemma}{Lemma}[section]
\newtheorem{Claim}{Claim}[section]
\newtheorem{Definition}{Definition}[section]
\newtheorem{Example}{Example}[section]
\newtheorem{Remark}{Remark}[section]

\newcommand\virt{\rm{virt}}
\newcommand\SO{\rm{SO}}
\newcommand\G{\varGamma}
\newcommand\Om{\Omega}
\newcommand\Kbar{{K\kern-1.7ex\raise1.15ex\hbox to 1.4ex{\hrulefill}}}
\newcommand\codim{\rm{codim}}
\renewcommand\:{\colon}
\newcommand\s{\sigma}
\def\vol#1{{|{\bfS}^{#1}|}}

\def\fa{{\mathcal F}}
\def\H{{\mathcal H}}
\def\O{{\mathcal O}}
\def\P{{\mathcal P}}
\def\L{{\mathcal L}}
\def\C{{\mathcal C}}
\def\Z{{\mathcal Z}}

\def\M{{\mathcal M}}
\def\N{{\mathcal N}}
\def\R{{\mathcal R}}
\def\ea{{\mathcal e}}
\def\Oa{{\mathcal O}}
\def\ee{{\bfE}}

\def\A{{\mathcal A}}
\def\B{{\mathcal B}}
\def\H{{\mathcal H}}
\def\V{{\mathcal V}}
\def\U{{\mathcal U}}
\def\al{{\alpha}}
\def\be{{\beta}}
\def\ga{{\gamma}}
\def\Ga{{\Gamma}}
\def\om{{\omega}}
\def\Om{{\Omega}}
\def\La{{\Lambda}}
\def\ov{\overline}
\def\dd{{\bfD}}
\def\pp{{\bfP}}

\def\nn{{\mathbb N}}
\def\zz{{\mathbb Z}}
\def\bq{{\mathbb Q}}
\def\bp{{\mathbb P}}
\def\bd{{\mathbb D}}
\def\bh{{\mathbb H}}
\def\te{{\theta}}
\def\rr{{\mathbb R}}
\def\bb{{\mathbb B}}

\def\pp{{\mathbb P}}

\def\dd{{\mathbb D}}
\def\zz{{\mathbb Z}}
\def\qq{{\mathbb Q}}

\def\hh{{\mathbb H}}
\def\nn{{\mathbb N}}

\def\LL{{\mathbb L}}

\def\co{{\mathbb C}}
\def\qq{{\mathbb Q}}
\def\na{{\mathbb N}}
\def\esima{${}^{\text{\b a}}$}
\def\esimo{${}^{\text{\b o}}$}
\def\lg{\lambdangle}
\def\rg{\rangle}
\def\ro{{\rho}}
\def\lV{\left\Vert}
\def\rV{\right\Vert }
\def\lv{\left\vert}
\def\rv{\right\vert }
\def\Sa{{\mathcal S}}
\def\D{{\mathcal D  }}

\def\si{{\bf S}}
\def\ve{\varepsilon}
\def\vr{\varphi}
\def\lV{\left\Vert }
\def\rV{\right\Vert}
\def\lv{\left\vert }
\def\rv{\right\vert}
\def\Range{\rm{{R}}}
\def\vol{\rm{{Vol}}}
\def\ind{\rm{{i}}}

\def\Int{\rm{{Int}}}
\def\Dom{\rm{{Dom}}}
\def\supp{\rm{{supp}}}
\def\Aff{\rm{{Aff}}}
\def\Exp{\rm{{Exp}}}
\def\Hom{\rm{{Hom}}}
\def\codim{\rm{{codim}}}
\def\cotg{\rm{{cotg}}}
\def\dom{\rm{{dom}}}
\def\Sa{\mathcal{{S}}}

\def\VIP{\rm{{VIP}}}
\def\argmin{\rm{{argmin}}}
\def\Sol{\rm{{Sol}}}
\def\Ker{\rm{{Ker}}}
\def\Sat{\rm{{Sat}}}
\def\diag{\rm{{diag}}}
\def\rank{\rm{{rank}}}
\def\Sing{\rm{{Sing}}}
\def\sing{\rm{{sing}}}
\def\hot{\rm{{h.o.t.}}}

\def\Fol{\rm{{Fol}}}
\def\grad{\rm{{grad}}}
\def\id{\rm{{id}}}
\def\Id{\rm{{Id}}}
\def\sep{\rm{{Sep}}}
\def\Aut{\rm{{Aut}}}
\def\Sep{\rm{{Sep}}}
\def\Res{\rm{{Res}}}
\def\ord{\rm{{ord}}}
\def\h.o.t.{\rm{{h.o.t.}}}
\def\Hol{\rm{{Hol}}}
\def\Diff{\rm{{Diff}}}
\def\SL{\rm{{SL}}}


\section{Introduction and main results}

In this paper we address the subject of stability  for a singularity of a holomorphic   vector field  in complex dimension two. We adopt the viewpoint of the  theory of Ordinary Differential Equations. More precisely, we introduce a natural notion of stable singularity inspired in the notion of Lyapunov, taking into account the holomorphic character: existence of a separatrix (\cite{camacho-sad}) and nonexistence of bounded solutions. In the sequel we consider  the global framework, by the use of theory of Foliations.

 In dimension two, there is a natural connection between vector field singularities and germs of holomorphic foliations  (\cite{mattei-moussu},\cite{scarduaDynI}).  We shall therefore refer to a {\it germ of a holomorphic foliation} $\fa$ at the origin $0 \in \mathbb C^2$, as induced by  a  pair $(X,U)$ where $X$ is a   holomorphic vector field defined a neighborhood $U$ of the origin $0\in \mathbb C^2$, with an isolated  singularity at the origin $X(0)=0$. Recall that a {\it separatrix} is an invariant irreducible  analytic curve containing the singularity. Every  singularity admits a separatrix (\cite{camacho-sad}). Throughout  this paper we will only consider singularities with a finite number of separatrices, called {\it non-dicritical}. In this case,  for any small enough neighborhood $U$ of the origin, the set of separatrices $\mbox{Sep}(\fa)\cap U$ is a non-empty finite union of analytic curves all of them containing the origin. Inspired by Reeb's local stability theorem for foliations (\cite{C-LN}) and by the classical notion of stability (in the sense of Lyapunov) for the solutions of a real ordinary differential equations, we introduce the following notion.   We shall say the germ $\fa$ is {\it $L$-stable} if, it is non-dicritical and, for any neighborhood $W$ of $\mbox{Sep}(\fa)$ there is  a smaller neighborhood $\mbox{Sep}(\fa)\subset V\subset W$ whose saturate by  $\fa$ is contained in $W$. We refer to Definition~\ref{Definition:Lstablelocal} for further details.
We shall then characterize those germs of ($L$-stable) foliations. A germ $\fa$ of a  {\it real logarithmic foliation} if it is  given by a closed meromorphic one-form $\Omega=\sum\limits_{j=1} ^r \lambda _j df_j /f_j$, for some holomorphic $f_j\in \mathcal O_2$ and $\lambda_j \in \mathbb R, \forall j$.
Ou main result reads:

\begin{Theorem}
\label{Theorem:main}
Let $\fa$ be a germ of a  $L$-stable singularity at $0\in {\mathbb C} ^2$.
  Then we have two possibilities:
   \begin{enumerate}
   \item $\fa$ admits a holomorphic first integral.
   \item $\fa$ is a real logarithmic foliation  singularity.
        \end{enumerate}

\end{Theorem}

For $L$-stable germs we can assure the existence of holomorphic first integral, by analyzing the {\it open leaves}, i.e., those which are not contained in the set of separatrices.

\begin{Corollary}
\label{Corollary:closedoff}
A non-dicritical germ of a foliation at the origin $0 \in \mathbb C^2$ admits a holomorphic first integral if, and only if, it is $L$-stable and has some open leaf which is not recurrent, e.g., if some open leaf is closed off the set of separatrices.
\end{Corollary}

    \section{Groups of complex diffeomorphisms}{}
Let  $\Diff({\mathbb C},0)$  denote the group of germs at the origin $0\in {\mathbb C}$ of holomorphic diffeomorphisms.  A germ $f \in G$ writes
$f(z)= e^{2 \pi \sqrt{-1} \lambda} z + a_{k+1} z^{k+1} + ...$.  The germ is {\it resonant} if $\lambda \in \mathbb Q$, its  dynamics is well-known (\cite{bracci,camacho}). We call $f$  {\it hyperbolic} if $\lambda \in \mathbb C \setminus \mathbb R$. A  hyperbolic germ of a diffeomorphism is {\it analytically linearizable} (\cite{Dulac}), and its dynamics is one of an attractor or of a repeller.  If  $\lambda \in \mathbb R \setminus \mathbb Q$  then we shall say that the diffeomorphism is {\it {\rm(}real{\rm)} non-resonant}. In the case the map is anaytically linearizable it looks like an irrational rotation. In the non linearizable case  P\'erez-Marco (\cite{P6,P7}) gives a quite accurate description of their rich dynamics. Such a map $f$ will also be referred to in this paper as a {\it P\'erez-Marco} map germ.

 The following definition will be useful.
\begin{Definition}[$L$-stable group]
\label{Definition:Lstablegroup}
{\rm
 A subgroup $G\subset \Diff(\mathbb C,0)$ is {\it $L$-stable}
if for any neighborhood $0 \in U\subset \mathbb C$ there is a sub-neighborhood $0 \in V \subset U \subset \mathbb C$ such that given  a point $p \in V$ its  pseudo-orbit (under the action of $G$) remains in $U$.}
\end{Definition}

A group $G\subset \Diff(\mathbb C,0)$ of germs of holomorphic diffeomorphisms will be called {\it resonant} if each map $g \in G$ is a resonant germ. The subgroup $G\subset \Diff(\mathbb C,0)$ is said to be of {\it circle type} group if it is  abelian and contains a real non-resonant analytically linearizable map. In particular, the group is analytically linearizable.

An important  tool is:

\begin{Lemma}
\label{Lemma:Lstablegroup}
Let $G\subset \Diff(\mathbb C,0)$ be a finitely generated  $L$-stable subgroup.
Then $G$ is finite cyclic or circle type. In particular,  any finitely generated resonant subgroup of a $L$-stable group is finite. The group $G$ is finite if  it admits a non-trivial pseudo-orbit which is closed off the origin (or closed). Also $G$ is finite provided that some non-trivial  pseudo-orbit is not recurrent.
\end{Lemma}
\begin{proof}
The group cannot contain flat elements $f(z)= z + a_{k+1} z^{k+1} + \ldots$ because such elements do not satisfy the $L$-stability condition (cf. \cite{camacho}). Thus $G$ is abelian and the subgroup $G_{\res}\subset G$ of resonant elements,  is finite cyclic. For the same reason $G$ cannot contain hyperbolic elements. Finally, a non-resonant map germ $f\in G$ must be analytically linearizable, thanks to the description of its dynamics given by P\'erez-Marco (\cite{P6,P7}).
This shows that either $G$ is finite cyclic or it is a product of a cyclic group and a linearizable group containing some irrational rotation.
As for the second part we observe that, since $G$ has a maximal cyclic finite subgroup, which is normal say generated by $g\in G$,
any two periodic maps in $G$ are powers of $g$ and therefore commute.
Finally, assume that $G$ has a pseudo-orbit which is closed off the origin. Then $G$ cannot contain irrational rotation, and therefore $G$ is finite. The same holds for the case $G$ exhibits some non-trivial pseudo-orbit which is not recurrent.
\end{proof}

\section{Reduction of singularities in dimension two (\cite{seidenberg})}
\label{section:reduction}
 Fix now a germ of holomorphic foliation with a singularity at the origin $0\in \mathbb C^2$.
 Choose a representative $\fa(U)$ for the germ $\fa$, defined in an open neighborhood $U$ of the origin, such that $0$ is the only singularity of $\fa(U)$ in $U$. In this framework, the
 {\it Theorem of reduction of singularities} (\cite{seidenberg})
  asserts the existence of a proper holomorphic map $\sigma\colon \tilde U \to U$  which is a finite composition of quadratic blowing-up's, starting with a blowing-up  at the origin,  such that the pull-back foliation
$\tilde \fa:= \sigma^* \fa$ of $\fa$ by $\sigma$ satisfies:

 \begin{enumerate}
 \item The {\it exceptional divisor\/} $E=\sigma^{-1}(0)\subset$
$\widetilde U$  can be written as
$E=\bigcup _{j=1}^m D_j$, where each irreducible component $D_j$ is diffeomorphic to an
embedded projective line ${\mathbb C} P(1)$ introduced as a divisor of
the successive blowing-up's (\cite{camacho-sad}).

\item  $\mbox{sing} \tilde \fa  \subset E$ is a finite set, and
 any singularity $\tilde p \in \sing \tilde \fa$ is {\it irreducible} i.e., belongs to one of
the following categories:
\begin{enumerate}
 \item  $xdy - \lambda ydx +\h.o.t.=0$ and $\lambda$ is not a positive rational number, i.e.
 $\lambda\notin\qq_+$
({\it non-degenerate singularity}),

 \item  $y^{p+1}dx -[x(1 + \lambda y^p)+ \h.o.t.]$ $dy=0$, $p\ge 1$.
This  case is called a {\it saddle-node\/} and it is extensively studied in  (cf. \cite{martinet-ramisselano}).
\end{enumerate}

\end{enumerate}

We call the lifted foliation $\tilde \fa$
the {\it desingularization} or {\it reduction of singularities} of
$\fa$.  We observe  that {\em $\fa$ is
non-dicritical iff $E$ is invariant by $\tilde \fa$}.
   Any two  components  $D_i$ and $D_j$, $i\ne j$, of the exceptional divisor,  intersect (transversely) at at
most one point, which is  called a {\it corner}. There are no
triple intersection points. An irreducible singularity $xdy - \lambda ydx + \hot=0$ is in the {\it Poincar\'e domain} if $\lambda \notin\mathbb R_-$ and it is in the {\it Siegel domain} otherwise.

We say that $\fa$ is a  {\it generalized curve\/} if its reduction of singularities produces only non-degenerate  (i.e., no saddle-node) singularities (\cite{C-LN-S1}).

\subsection{Holonomy and virtual holonomy groups}
\begin{center}

\end{center}
Let now $\fa$ be a holomorphic foliation with (isolated)  singularities on a complex
surface $M$ (we have in mind here, the result of a reduction of singularities process). Denote by $\sing(\fa)$ the singular set of $\fa$. Given a leaf $L_0$ of $\fa$ we choose any base point $p\in L_0\subset M \setminus \sing(\fa)$ and a transverse disc $\Sigma_p\Subset M$ to $\fa$ centered at $p$. The holonomy group of the leaf $L_0$ with respect to the disc $\Sigma_p$ and to the base point $p$ is denoted by  $\mbox{Hol}(\fa,L_0,\Sigma_p,p)\subset \mbox{Diff}(\Sigma_p,p)$. By considering any parametrization $z\colon (\Sigma_p,p) \to (\mathbb D,0)$ we may identify
(in a non-canonical way) the holonomy group with a subgroup of $\Diff(\mathbb C,0)$. It is clear from the construction  that the maps in the holonomy group preserve the leaves of the foliation. Nevertheless, this property can be shared by a larger group that may therefore contain more information about the foliation in a neighborhood of the leaf.
The {\it virtual holonomy group} of the leaf with respect to the transverse section $\Sigma_p$ and base point $p$ is defined as (\cite{C-LN-S}, \cite{camacho-scarduaasterisque})
$$\mbox{Hol}^{\mbox{virt}}(\fa,\Sigma_p,p)=\{ f \in\mbox{Diff}
(\Sigma_p, p) \big| L_z =  L_{f(z)},
\forall z\in (\Sigma_p, p) \}$$

 Fix now a germ of holomorphic foliation with a singularity at the origin $0\in \mathbb C^2$, with a  representative $\fa(U)$ as above. Let $\Gamma$ be a separatrix of $\fa$. By Newton-Puiseaux parametrization theorem,  $\Gamma\setminus \{0\}$  is
biholomorphic to a punctured disc $\mathbb D^*= \mathbb D\setminus \{0\}$. In particular,  we may choose a
loop $\gamma\in \Gamma\setminus\{0\}$ generating the (local) fundamental group $\pi_{1}(\Gamma\setminus\{0\})$. The corresponding holonomy map $h_{\gamma}$ is defined in terms of a germ of complex diffeomorphism at the origin of a local disc $\Sigma$ transverse to $\mathcal{F}$ and centered at a non-singular point $q\in \Gamma\setminus\{0\}$. This map is well-defined up to conjugacy by germs of holomorphic diffeomorphisms, and is generically referred to as \textit{local holonomy} of the separatrix $\Gamma$.

\subsection{$L$-stable singularities}
\begin{center}

\end{center}

Let $\fa,U$ be as in the beginning of this section. 
\begin{Definition}
\label{Definition:Lstablelocal}
{\rm Assume that  $\fa$ is a non-dicritical germ.  Denote by $\mbox{Sep}(\fa,U)$ the (analytic) set of separatrices of $\fa(U)$ in $U$.  We say that $\fa(U)$ is {\it $L$-stable} if given any neighborhood
$W\subset U$ of $\mbox{Sep}(\fa,U)$ there is a sub-neighborhood $\mbox{Sep}(\fa,U) \subset V \subset W$ such that its saturation $\mbox{Sat}^\fa _W(V)$ by $\fa(U)\big|_W$ is still contained in $W$, i.e., $\mbox{Sat}^\fa _W(V)\subset W$. The germ $\fa$ is called {\it $L$-stable} if $\fa(U)$ is $L$-stable for every arbitrarily small neighborhood  $U$ of the origin.
}
\end{Definition}

Then we have:

\begin{Lemma}
\label{Lemma:separatrix}
Let $\fa$ be a germ of a non-dicritical holomorphic foliation with a singularity at the origin $0\in \mathbb C^2$ and let $\Gamma$ be a separatrix of $\fa$.
If $\fa$ is $L$-stable then the holonomy map $h_\gamma\in \Diff(\mathbb C,0)$ is $L$-stable in the sense of Definition~\ref{Definition:Lstablegroup}.
\end{Lemma}
The proof is   standardly  obtained from a simple adaptation in the proof of Lemma~\ref{Lemma:holonomystable}.
\subsection{The non-degenerate (irreducible) case}
\begin{center}

\end{center}
Let us consider  a germ of a holomorphic foliation $\fa$ at the origin $0\in \mathbb C^2$,  a germ of an irreducible non-degenerate   singularity. In suitable local coordinates we can write $\fa$ as given by
$x(1+ A(x,y)) dy - \lambda y(1 + B(x,y))dx=0$, for some holomorphic $A(x,y), \, B(x,y)$ with $0 \ne \lambda \in \mathbb C \setminus \mathbb Q_+, \, A(0,0)=B(0,0)=0$. In the normal form above, the separatrices are the coordinate axes. Let us denote by $f$ the holonomy map (its class up to holomorphic conjugacy) of the separatrix $(y=0)$. From the correspondence between the leaves of $\fa$ and the orbits of $f$ (\cite{mattei-moussu,martinet-ramisresonante,perezmarcoyoccoz}) and according to the well-known properties of $f$ (see \cite{bracci,camacho} for the case of germs with periodic  linear part and \cite{P6,P7} for the case of non-resonant germs) we have:
\begin{Lemma}
\label{Lemma:irreduciblenondegenerate}
Let $\fa$ be a germ of an irreducible $L$-stable  singularity at the origin $0 \in \mathbb C^2$. Then $\fa$ is  analytically linearizable of the form $xdy - \lambda ydx=0$ where $\lambda \in \mathbb Q_- \cup (\mathbb R\setminus \mathbb Q)$.
\end{Lemma}
\begin{proof}
We have in mind Lemma~\ref{Lemma:separatrix}. 
First we observe that the singularity is not a saddle-node. Indeed,  the strong  manifold of a saddle-node exhibits a non-trivial holonomy tangent to the identity, say of the form $z \mapsto z + a_{k+1} z^{k+1} + \ldots$ for some $a_{k+1} \ne 0, k \in \mathbb N$. Thus, a strong manifold  separatrix is not $L$-stable.  Therefore  the singularity is non-degenerate, of the form $xdy - \lambda ydx + ...=0$ for some $\lambda \in \mathbb C \setminus \mathbb Q_+$.
Analyzing case by case we have:
\begin{enumerate}
\item $\lambda \in \mathbb C \setminus \mathbb R$ hyperbolic case: In this case the singularity is analytically linearizable.
     We may therefore choose local coordinates $(x,y)\in (\mathbb C^2,0)$ such that the germ writes as $xdy - \lambda ydx=0$.
The holonomy of one of the coordinate axes with respect to a small disc $\Sigma: \{x=a\}$ is given by $h(y)=\exp(2 \pi \sqrt{-1}\lambda) y$.  Thus  $|h^\prime(0)|\ne 1$ and the map $h$ is not $L$-stable.
  \item  resonant  non-linearizable case, $\lambda \in \mathbb Q_-$:
      this case also cannot occur, because a holonomy map of a separatrix is then of the form $h(z)= e^{2 \pi i n /m}z + h_2 z^2 + h_3 z^3 +...$, for some $n ,m \in \mathbb N, \, <n,m>=1$.
      This map is not periodic, because the singularity is not linearizable (any periodic germ of diffeomorphism is linearizable). On the other hand $h^m = h\circ ... \circ h$ is (non-trivial and) tangent to the identity, and therefore $h$ is not $L$-stable.

      \item non-resonant  case, $\lambda \in \mathbb R \setminus \mathbb Q$: If the singularity is in the Poincar\'e domain, i.e., $\lambda \in \mathbb R_+$ then, since it is not a resonance (because $\lambda, 1/\lambda \notin\mathbb N$) it is analytically linearizable. This case can occur. Assume now that the singularity is in the Siegel domain, $\lambda \in \mathbb R_-\setminus \mathbb Q$. Suppose by contradiction that the singularity is not analytically linearizable. Then the local holonomy map of any separatrix is not analytically linearizable as well, that is, this holonomy map is a non-resonant nonlinearizable (P\'erez-Marco type map). In particular, we may conclude that this map is not $L$-stable (see Lemma~\ref{Lemma:Lstablegroup}), yielding a contradiction.

          \item resonant linearizable case: in this case the foliation is $L$-stable, it admits a holomorphic first integral.
\end{enumerate}

\end{proof}

\section{Reduction of stable  singularities}
The following notion is useful in our framework.
\begin{Definition}[$L$-stable divisor]
\label{Definition:L-stabledivisor}
{\rm Let $\fa$ be a holomorphic foliation with isolated singularities on a complex surface $M$ and $D\subset M$ an invariant compact connected analytic subset of pure dimension one. Assume that the singularities of $\fa$ in $D$ are non-dicritical. Under these conditions given a small enough neighborhood $D \subset U\subset M$, the set of local separatrices of $\fa(U):=\fa\big|_U$ though singular points in $D$ together with $D$ gives a pure dimension one analytic subset of $U$ denoted by $\mbox{Sep}(\fa(U),D) \cup D\subset U$. We shall then say that $D$ is {\it $L$-stable} (with respect to $\fa$)  if for any arbitrarily small neighborhood $U$ as above,  given any neighborhood $W\subset U$ of $\mbox{Sep}(\fa(U),D) \cup D\subset U$ there is a sub-neighborhood $\mbox{Sep}(\fa,U) \subset V \subset W$ such that its saturation $\mbox{Sat}^\fa _W(V)$ by $\fa(U)\big|_W$ is still contained in $W$, i.e., $\mbox{Sat}^\fa _W(V)\subset W$.}
\end{Definition}

\begin{Lemma}
\label{Lemma:zero}
Let $\fa$ be a  $L$-stable germ of a foliation. Then $\fa$ is a generalized curve.
\end{Lemma}
\begin{proof}

We proceed by induction on the number $r\in \{0,1,2,...\}$ of blowing-ups
in the reduction of singularities for the germ $\fa$.

\noindent{\bf Case 1}. ($r=0$). In this case the singularity is already irreducible. The result follows from Lemma~\ref{Lemma:irreduciblenondegenerate}.

\noindent{\bf Case 2} (Induction step). Assume that the result is proved for foliation germs that admit a reduction
of singularities with a number of blowing-ups less greater than or equal to $r$. Suppose that the fixed
germ $\fa$ admits a reduction of singularities consisting of $r+1$ blowing-ups.
We fix a small  neighborhood $U\subset \mathbb C^2$ of the origin, such that $\mbox{Sep}(\fa,U)$ is analytic of pure dimension one.
For simplicity of the notation we shall denote $\fa(U)$ by $\fa$.

Then we perform a first blow-up $\sigma_1\colon \tilde U (1) \to U$ at the origin and obtain a lifted
foliation $\tilde \fa(1) = \sigma_1^*(\fa)$ with  (first) exceptional divisor $E(1)=\sigma_1^{-1}(0)$ consisting of
a single embedded invariant projective line in $\tilde U(1)$ (by hypothesis  the exceptional divisor is invariant by $\tilde \fa (1)$).
For any neighborhood $0 \in V\subset U\subset \mathbb C^2$ of the set of separatrices $\mbox{sep}(\fa,U)$, the inverse image $\tilde V(1):= \sigma_1^{-1}(V)\subset \tilde U(1)$  is a neighborhood of $E(1)\cup \mbox{Sep}(\fa(\tilde U(1)),E(1))$. Moreover, $\tilde V(1)$ is $\tilde \fa (1)\big|_{\tilde U(1)}$-invariant if and only if $V$ is $\fa\big|_U$-invariant.  Given a leaf $L$ of $\fa$ in $U$ we denote by $\tilde L (1) $ the lifting $\tilde L(1)=\sigma_1^{-1}(L)$ of $L$ to $\tilde U(1)$ by the map $\sigma_1\colon \tilde U (1) \to U$.
Notice that for each singularity $\tilde p \in \mbox{sing}(\tilde \fa(1))\subset E(1)$ the set of local separatrices of $\tilde \fa(1)$ through $\tilde p$ is formed by $E(1)$ union the local branches through $\tilde p$, of the strict transform by $\sigma(1)$ of  $\mbox{Sep}(\fa,U)$.
From the fact that the germ $\fa$ is $L$-stable, we conclude that:
\begin{Claim}
The divisor $E(1)$ is $L$-stable with respect to $\tilde \fa(1)$.
\end{Claim}
Given now a singularity $\tilde p \in \mbox{sing}(\tilde \fa(1))\cap E(1)$ we observe that the germ of $\tilde \fa(1)$ at $\tilde p$ is $L$-stable.
Therefore, by the induction hypothesis, $\tilde p$ is a generalized curve. This shows that all singularities in $E(1)$ are generalized curves and therefore the germ $\fa$ is a generalized curve. The lemma follows from the Induction Principle.
\end{proof}

\section{$L$-stable germs of foliations}
\label{section:germ}

In this section we prove  Theorem~\ref{Theorem:main}. We rely on Lemma~\ref{Lemma:zero} and on Lemma~\ref{Lemma:a} below.

\begin{Lemma}
\label{Lemma:a}
Let $\fa$ be a non-dicritical $L$-stable foliation germ as in Theorem~\ref{Theorem:main}. Then the following are equivalent:
\begin{enumerate}
\item $\fa$ admits a  holomorphic first integral in some neighborhood of the origin.
\item $\fa$  is  fully-resonant $\mbox{(see \cite{scarduaDynI})}$.
    \end{enumerate}

\end{Lemma}
\begin{proof}
Since (1) implies (2) is well-known, we prove the converse. Assume then that all final singularities in the reduction process are resonant.

We proceed by induction on the number $r\in \{0,1,2,...\}$ of blow-ups
in the reduction of singularities for the germ $\fa$.

\noindent{\bf Case 1}. $(r=0)$. This case is a consequence of Lemma~\ref{Lemma:irreduciblenondegenerate}.

\noindent{\bf Case 2}. $(r-1 \implies r)$.  Assume that the result is proved for foliation germs that admit a reduction of singularities with a number of blow-ups smaller than $r$. Suppose that the fixed germ $\fa$ admits a reduction of singularities consisting of $r$ blow-ups. Let $U$ be a small connected neighborhood of the origin where the leaves of $\fa$ are closed off the set of separatrices. Then we proceed as in the proof of Lemma~\ref{Lemma:zero} from where we import the notation.
Thus we perform a first blow-up $\sigma_1\colon \tilde U (1) \to U$ at the origin and obtain a lifted
foliation $\tilde \fa(1) = \sigma_1^*(\fa)$ with  (first) exceptional divisor $E(1)=\sigma_1^{-1}(0)$ consisting of
a single embedded invariant projective line in $\tilde U(1)$ (by hypothesis  the exceptional divisor is invariant by $\tilde \fa (1)$). Given a leaf $L$ of $\fa$ in $U$ we denote by $\tilde L (1) $ the lifting $\tilde L(1)=\sigma_1^{-1}(L)$ of $L$ to $\tilde U(1)$ by the map $\sigma_1\colon \tilde U (1) \to U$. From the proof of Lemma~\ref{Lemma:zero} we know that {\em the divisor $E(1)$ is $L$-stable with respect to $\tilde \fa(1)$.}

 By the Induction hypothesis, each singularity $\tilde p\in \sing(\tilde \fa)$ admits a holomorphic first integral say $\tilde f_{\tilde p}$ defined in $\tilde W_{\tilde p}$ if this last is small enough.
Now we analyze the holonomy of the leaf $E_0:=E \setminus \sing(\tilde \fa)$. Choose a regular point $\tilde q \in E_0$ and a small transverse disc $\Sigma$ to $E_0$ centered at $\tilde q$. The corresponding holonomy group representation will be denoted by $H:=\mbox{Hol}(\tilde \fa, \Sigma, \tilde q)\subset \mbox{Diff}(\Sigma, \tilde q)$. This holonomy group is finitely generated.  By the invariance of $E$ and a natural version of Lemma~\ref{Lemma:separatrix}, we know that  {\sl the  holonomy group $H$ is $L$-stable}. Since this holonomy group is generated by periodic maps, applying Lemma~\ref{Lemma:Lstablegroup}  we conclude that the holonomy group is finite. Since the virtual holonomy group preserves the leaves of the foliation, the arguments above already show that
{\sl the virtual holonomy group $H^{\virt}$ of the exceptional divisor is $L$-stable}. The problem is that we still do not know that the virtual holonomy group is resonant so that we cannot conclude that this virtual holonomy group is finite.  Nevertheless, from Lemma~\ref{Lemma:Lstablegroup} we obtain:

\begin{Claim}
\label{Claim:subgroupvirtual}
Any finitely generated resonant subgroup $H_0$ of the virtual holonomy group
$H^{\virt}$ is a finite group.
\end{Claim}

Let us then proceed as follows:
given the singularities $\{ \tilde p_1, ...,\tilde p_m\}=\mbox{sing}(\tilde \fa) \subset E$, by induction hypothesis each singularity admits a local holomorphic first integral. Thus, there are small discs $D_j\subset E$, centered at the $\tilde p_j$ and such that in a neighborhood $V_j$ of $\tilde p_j$ in the blow-up space
$\tilde{\mathbb C^2 _ 0}$ , of product type $V_j = D_ j\times \mathbb D_\epsilon$, we have a holomorphic first integral $g_j \colon  V_j \to \mathbb C$, with $g_j(\tilde p_j)=0$. Fix now a point $\tilde p_0\in E \setminus \sing(\tilde \fa)$. Since $E$ has the topology of the $2$-sphere, we may choose a simply-connected domain $A_j \subset E$ such that $A_j\cap \{\tilde p_0,\tilde p_1,...,\tilde p_m\}=\{\tilde p_0, \tilde p_j\}$, for every $j=1,...,m$. Since $A_j$ is simply-connected, we may extend the local holomorphic first integral $g_j$ to a holomorphic first integral $\tilde g_j$ for $\tilde \fa$ in a neighborhood $U_j$ of $D_j \cup A_j$, we may assume that $U_j$ contains $V_j$. Now, given a local transverse section $\Sigma_0$ centered at $\tilde p_0$ and contained in $U_j$, we may introduce the {\it invariance group} of the restriction $g_j ^0 := \tilde  g_j\big|_{\Sigma_0}$ as the group $\mbox{Inv}(g_j ^0):= \{ f \in \mbox{Diff}(\Sigma_0, p_0), \, g_j ^0\circ f= g_j ^0\}$. In other words, the invariance group of $g_j ^0$ is the group of germs of maps that preserve the fibers of $g_j ^0$. Clearly $\mbox{Inv}(g_j ^0)$ is a finite (resonant) group (\cite{mattei-moussu} Proposition 1.1. page 475). Let us now denote by $\mbox{Inv}(\tilde \fa, \Sigma_0)\subset \mbox{Diff}(\Sigma_0,p_0)$ the subgroup generated by the invariance groups $\mbox{Inv}(g_j ^0), j=1,...,m$. We call $\mbox{Inv}(\tilde \fa, \Sigma_0)$ the {\it global invariance group} of $\tilde \fa$ with respect to $(\Sigma_0,p_0)$.  Then, from the above we immediately obtain:
\begin{Claim}
\label{Claim:globalinvariance}
$\Inv(\tilde \fa, \Sigma_0)$ is a finite group.

\end{Claim}
\begin{proof}
Indeed, first notice that $\Inv(\tilde \fa, \Sigma_0)$     generated by periodic maps. Since $\Inv(\tilde \fa, \Sigma_0)$ preserves the leaves of $\tilde \fa$ we that $\mbox{Inv}(\tilde \fa, \Sigma_0)\subset \mbox{Hol}^{\virt} (\tilde \fa, \Sigma_0, p_0)$ and therefore $\Inv(\tilde \fa, \Sigma_0)$ is  a finite group.
\end{proof}

Notice that this global invariance group contains in a natural way the local invariance groups of the local first integrals $g_j$. Therefore, as observed in \cite{mattei-moussu}, once we have proved that the global invariance group $\Inv(\tilde \fa, \Sigma_0)$ is finite, together with the fact that the singularities in $E$ exhibit local holomorphic first integrals, we conclude as in \cite{mattei-moussu} that the foliation $\tilde \fa$ and therefore the foliation $\fa$ has a holomorphic first integral.

\end{proof}

\begin{proof} [Proof of Theorem~\ref{Theorem:main}]
By hypothesis $\fa$ is a $L$-stable foliation germ, with a representative defined in a  neighborhood $U$ of the origin.  By Lemma~\ref{Lemma:zero} $\fa$ is a generalized curve. By Lemma~\ref{Lemma:a} we may assume that the germ is not fully-resonant. The reduction of singularities of $\fa$ exhibits some non-resonant  singularity. 

\begin{Claim} In the above situation we have:
\begin{enumerate}
\item Given any separatrix $\Gamma$ through the origin, and a transverse disc $\Sigma$ meeting $\Gamma$ at a point $0\ne q=\Sigma \cap \Gamma$, the  virtual holonomy group $\mbox{Hol}^{\;\virt}(\fa,\Sigma, q)$ is an abelian linearizable circle type group containing a non-resonant  map.
\item In particular, there is a closed meromorphic 1-form $\Omega$ such that:
\begin{enumerate}

\item $\Omega$ is meromorphic with simple poles defined in a neighborhood $0 \in V \subset U$ of the singularity.

\item The foliation is defined by $\Omega=0$, off its polar set $(\Omega)_\infty \subset V$.

\item For any holomorphic coordinate $z \in \Sigma$ that linearizes the virtual holonomy group
$\Hol^{\virt}(\tilde \fa,\Sigma)$ we have $\Omega\big|_{\Sigma}= dz /z$.

\end{enumerate}

    \end{enumerate}
\end{Claim}

\begin{proof}[Proof of the claim]

We proceed by induction on the number $r\in \{0,1,2,...\}$ of blowing-ups
in the reduction of singularities for the germ $\fa$.

\noindent{\bf Case 1}. ($r=0$). In this case the singularity is already irreducible and non-resonant. The result follows from Lemma~\ref{Lemma:irreduciblenondegenerate}.

\noindent{\bf Case 2} (Induction step). Assume that the result is proved for (not fully-resonant) foliation germs that admit a reduction
of singularities with a number of blowing-ups less greater than or equal to $r$. Suppose that the fixed
germ $\fa$ admits a reduction of singularities consisting of $r+1$ blowing-ups. We fix a small  neighborhood $U\subset \mathbb C^2$ of the origin, such that $\mbox{Sep}(\fa,U)$ is analytic of pure dimension one.
For simplicity of the notation we shall denote $\fa(U)$ by $\fa$.
Write $\mbox{sing}(\tilde(\fa))=\{\tilde p_1,...,\tilde p_\ell\}$.
We know from the proof of Lemma~\ref{Lemma:zero} that the germ of foliation induced by $\tilde \fa(1)$ at each singular point $\tilde p_j$ is $L$-stable. If the singularity $\tilde p_j$ is fully-resonant then it admits a holomorphic first integral by Lemma~\ref{Lemma:a}. Thus, in this case, there is a closed meromorphic one-form $\tilde \Omega_j$ with simple poles defining $\tilde \fa(1)$ in a neighborhood $\tilde W_j$ of $\tilde p_j$. Assume now that $\tilde p_j$ is not fully-resonant. Then,  by the Induction hypothesis we have: 
\begin{enumerate}
\item The  reduction of singularities of $\tilde p_j$ only produces linearizable singularities, with real quotient of eigenvalues.
\item Given any separatrix $\tilde \Gamma(j)$ through the $\tilde p_j$, and a transverse disc $\tilde \Sigma_j$ meeting $\tilde \Gamma(j)$ at a point $\tilde p_j\ne \tilde q_j=\tilde \Sigma_j \cap \tilde \Gamma(j)$, the  virtual holonomy group $\mbox{Hol}^{\virt}(\tilde \fa(1),\tilde \Sigma_j, \tilde q_j)$ is an abelian linearizable circle type group containing a non-resonant map.
\item There is a closed meromorphic one-form $\tilde \Omega_j$  defining $\tilde \fa(1)$ in a neighborhood $\tilde W_j$ of $\tilde p_j$ and such that:
    \begin{enumerate}

\item $\tilde \Omega$ has simple poles in $\tilde W_j$.
\item $\tilde \Omega$ defines $\tilde \fa(1)$ in $\tilde W_j\setminus (\tilde \Omega_j)_\infty$.
\item Given a point $\tilde q_j \in E(1)\setminus\sing(\tilde \fa(1))$, close to the singularity $\tilde p_j$, and a transverse disc $\Sigma_{\tilde q_j}$ to $E(1)$,  and  any holomorphic coordinate $z_j \in \Sigma_{\tilde q_j}$ that linearizes the virtual holonomy group $\mbox{Hol}^{\virt}(\tilde \fa(1)_{\tilde p_j},\Sigma_{\tilde q_j})$, of the germ of $\tilde \fa(1)$ at the point $\tilde p_j$,  then $\tilde \Omega_j\big|_{\Sigma_{\tilde q_j}}= dz_j /z_j$.
    \end{enumerate}
\end{enumerate}
Now, also from the proof of Lemma~\ref{Lemma:zero} we know that the {\em the divisor $E(1)$ is $L$-stable with respect to $\tilde \fa(1)$.} Choose now a point
$\tilde q \in E(1)\setminus \sing(\tilde \fa(1))$ and a local transverse disc $\Sigma_{\tilde q} \ni \tilde q$. Then the virtual holonomy group
$\mbox{Hol}^{\virt} (\tilde \fa(1), \Sigma_{\tilde q}, \tilde q)$ is $L$-stable.
Because some singularity $\tilde p_j$ is not fully resonant, this group contains some non-resonant element, and therefore it is analytically linearizable. Once we have this we can proceed as in \cite{C-LN-S}.
Indeed, since $E(1)$ is invariant, the non-resonant (linearizable) element in the virtual holonomy of a separatrix through $\tilde p$  induces a non-resonant linearizable element on the virtual holonomy of the separatrix through $\tilde q$  induced by the exceptional divisor $E(1)$.
This is done as follows. Given two points $\tilde q$ and $\tilde q_1$, close to $\tilde p$ and $\tilde p_1$ respectively, and transverse discs $\Sigma$ and $\Sigma_1$ meeting $E(1)$ at these points respectively, we can choose a simple path $\alpha\colon [0,1] \to E(1) \setminus \sing(\tilde \fa (1))$ from $\tilde q$ to $\tilde q_1$. The holonomy map
$h_\alpha\colon (\Sigma,\tilde q) \to (\Sigma_1, \tilde q_1)$ associated to the path $\alpha$ (recall that $E(1)\setminus \sing(\tilde \fa (1))$ is a leaf of $\tilde \fa (1)$),
induces a natural morphism for the virtual holonomy groups
$\alpha ^* \colon \mbox{Hol}^{\virt} (\tilde \fa (1), \Sigma_1, \tilde q_1) \to \mbox{Hol}^{\virt}(\tilde \fa (1), \Sigma, q)$, by $\alpha^*:  h \mapsto h_{\alpha}^{-1} \circ h \circ h_{\alpha}$. Since $h_{\alpha^{-1}}=(h_{\alpha})^{-1}$ in terms of holonomy maps, we conclude that the above morphism is actually an isomorphism between the virtual holonomy groups. Thus, also the virtual holonomy group
associated to the separatrix $\tilde\Gamma$ of $\tilde \fa (1)$ through $\tilde p_1$
contains some non-resonant linearizable  map.

 The fact that the virtual holonomy of $E(1)$ is analytically linearizable gives a closed meromorphic one-form $\tilde \Omega(1)$ defined in a neighborhood $\tilde W(1)^*$ of the leaf $E(1)^*:= E(1)\setminus \sing(\tilde \fa(1))$ with the following property:
given a transverse disc $\Sigma$ to $E(1)^*$ with a holomorphic coordinate $z \in \Sigma$ that linearizes the virtual holonomy group
$\Hol^{\virt}(\tilde \fa(1),\Sigma)$, then $\tilde \Omega(1)\big|_{\Sigma}= dz /z$. Since the virtual holonomy group $\Hol^{\virt}(\tilde \fa(1),\Sigma)$ contains in a natural way the virtual holonomy group
$\mbox{Hol}^{\virt}(\tilde \fa(1)_{\tilde p_j},\Sigma_{\tilde q_j})$ this implies that a coordinate that linearizes the virtual holonomy group $\Hol^{\virt}(\tilde \fa(1),\Sigma)$ also linearizes the virtual holonomy group  $\mbox{Hol}^{\virt}(\tilde \fa(1)_{\tilde p_j},\Sigma_{\tilde q_j})$. This shows that $\tilde \Omega(1)$ and $\tilde \Omega_j$ coincide in a neighborhood of $\tilde q_j$, and therefore $\tilde \Omega(1)$ extends to a neighborhood of $\tilde p_j$ as $\tilde \Omega_j$. Thus we have constructed the form $\tilde \Omega(1)$ in a neighborhood $\tilde W(1)$ of $E(1)$ in $\tilde U(1)\subset \tilde {\mathbb C_0^2}$. Now we  project this form into a one-form $\Omega= (\sigma(1)^{-1})^*(\tilde \Omega(1))$ in $\sigma(\tilde U(1))\setminus\{0\}$, and extend $\Omega$ to the origin by Hartogs' theorem. Now we are about to complete the proof of the claim. For this, consider any separatrix $\Gamma$ of $\fa$ through the origin. Since the projective line $E(1)$ in the first blow-up is invariant, the lift $\tilde \Gamma$ is the separatrix of some singularity $\tilde p_1$ of $\tilde \fa (1)$. If $\tilde p_1$ is not fully-resonant, then by the above, we conclude that the virtual holonomy group associated to this separatrix $\Gamma$ contains a non-resonant map.  Assume now that $\tilde p_1$ is fully-resonant. In this case we have to show the existence of a non-resonant map  in the virtual holonomy of the separatrix $\tilde \Gamma$ by ``importing`` this map from some virtual holonomy of other singularity. Indeed, by hypothesis, some singularity $\tilde p$ in the first blow-up is not fully-resonant. Therefore, its virtual holonomy relatively to the separatrix contained in the projective line, contains  a non-resonant map.  Recall that the blow-up is a diffeomorphism off the origin and off the exceptional divisor, so that the maps in the virtual holonomy of $\tilde \Gamma$ induce maps in the disc $\Sigma$ transverse to $\Gamma$ in $\mathbb C^2$, but which are defined only in the punctured disc, i.e., off the origin. Nevertheless, since these projected maps are one-to-one, the classical Riemann extension theorem for bounded holomorphic maps shows that indeed such maps induce germs of diffeomorphisms defined in the disc $\Sigma$. These diffeomorphisms are the virtual holonomy maps of the separatrix $\Gamma$ of $\tilde \fa (1)$ evaluated at the transverse section $\Sigma$. Hence,  by projecting the maps in $\mbox{Hol}^{\virt}(\tilde \fa (1), \Sigma, \tilde q)$ we obtain non-resonant maps in this virtual holonomy group as stated. Now the Induction Principle applies to finish the proof of the claim.
\end{proof}

By the above claim we can construct a simple poles closed meromorphic one-form that defines $\fa$ in a neighborhood of the origin. By the Integration lemma (\cite{scarduaDynI}) $\omega$ is a logarithmic one-form $\omega=\sum\limits_{j=1}^r \lambda_j df_j/f_j$, where $\lambda_j\in \mathbb C\setminus \{0\}$. Denote by  $\Gamma_j$ the separatrix given by $\{f_j=0\}$. Given a coordinate $z_1\in \Sigma_1$, a transverse disc to $\Gamma_1$ at $p_1 \in \Gamma_1\setminus \{0\}$,  according to \cite{C-LN-S} the maps $z_1\mapsto \lambda_j /\lambda_1 z_1$ belong to the virtual holonomy group $\mbox{Hol}^{\virt}(\fa, \Sigma_1,p_1)$. Therefore,  we conclude that $\lambda _ j / \lambda_1 \in \mathbb R, \forall j$ and in the non-resonant case we must have  $\lambda _j / \lambda_1 \in \mathbb R \setminus \mathbb Q$ for some $j$. This ends the proof of Theorem~\ref{Theorem:main}.
 \end{proof}

 \begin{proof}[Proof of Corollary~\ref{Corollary:closedoff}]
From the last part of the proof above we know that the maps
 $z_1\mapsto \lambda_j/ \lambda_1 z_1$ belong to the virtual holonomy group $\mbox{Hol}^{\virt}(\fa, \Sigma_1,p_1)$. By hypothesis $\fa$ has a leaf which is closed off the set of separatrices, thus this virtual holonomy group must have a pseudo-orbit which is closed off the origin and therefore it must be resonant. Hence, still under the notation of the proof of Theorem~\ref{Theorem:main} we must have $\lambda _j / \lambda_1 \in \mathbb Q$ for all $j$. Thus $\fa$ admits a holomorphic first integral. This proves Corollary~\ref{Corollary:closedoff}.  Similar arguments prove the case where there is an open leaf which is  not recurrent. \end{proof}

\section{Global framework}
In this section  all manifolds are assumed to be connected and the foliations are non-singular. Let $\fa$ be a
 holomorphic foliation on a complex manifold $M$.

\begin{Definition}
{\rm A leaf $L_0 \in \fa$ is {\it $L$-stable} if
given any neighborhood $W$ of $L_0$ in $M$,  there is a neighborhood $L_0\subset V \subset W$
such that $\mbox{Sat}_\fa(V) \subset W$.

}
\end{Definition}

\begin{Theorem}
\label{Theorem:global}
Let $\fa$ be a codimension one holomorphic foliation on a complex manifold $M$. Let $L_0 \in \fa$ be a  compact leaf which is $L$-stable. Then
we have the following possibilities:
\begin{itemize}

\item[{\rm(i)}] $L_0$ has a finite holonomy.
\item[{\rm(ii)}] $L_0$ has a circle type  holonomy.
\end{itemize}
In the first case if $M$ is compact then all leaves are compact with finite holonomy group. In the second case, there is a $\fa$-saturated neighborhood of $L_0$ where $\fa$ is given by a closed meromorphic one-form $\Omega$ with polar set of order one, $(\Omega)_\infty=L_0$. In this case, if $M\setminus L_0$ is Stein, then $\Omega$ extends to $M$, closed meromorphic, with simple poles, $(\Omega)_\infty= L_0$. In particular, all other leaves of $\fa$ have trivial holonomy.
\end{Theorem}

\subsection{$L$-stable compact leaves}

Let us now give the proof of
Theorem~\ref{Theorem:global}. Let $\fa$ be a codimension one holomorphic foliation on a complex manifold $M$. Let $L_0 \in \fa$ be a  compact leaf which is $L$-stable.
Given a point $p\in L_0$ and a small transverse disc $\Sigma$ centered at $p$ we claim:
\begin{Lemma}
\label{Lemma:holonomystable} The virtual holonomy $\mbox{Hol}^{\;\virt}(\fa,L_0,\Sigma,p)$ of the leaf $L_0$, is $L$-stable.
\end{Lemma}
\begin{proof}
Indeed, given a small connected neighborhood $p \in U\subset \Sigma$ we consider the saturation $W:=\mbox{Sat}_\fa(U)\subset M$. This is an open neighborhood of $L_0$ in $M$. Denote by $U_p\subset W \cap \Sigma$ the connected component of $W \cap \Sigma$ that contains the point $p$. Because $U$ is connected, we have $U_p \subset U$. Then we put $U_1:=\mbox{Sat}_\fa(U_p)\cap \Sigma$. We observe that $\mbox{Sat}_\fa(U_1) \cap \Sigma \subset \mbox{Sat}_\fa(U_p) \cap \Sigma=U_1$, because $\mbox{Sat}_\fa(U_p)$ is already saturated. On the other hand, clearly $\mbox{Sat}_\fa(U_1) \cap \Sigma \supset U_1\cap \Sigma= U_1$, so that
 \[
 \mbox{Sat}_\fa(U_1) \cap \Sigma = U_1
 \]
 Thus, for sake of simplicity we will assume that $U_p$ is connected and
 $W=\mbox{Sat}_\fa(U_p)$ is such that $W\cap \Sigma= U_p$. Finally, since $U_p \subset U$ we may consider that $U_p=U$, so that $\mbox{Sat}_\fa(U) \cap \Sigma=U$.

 Since by hypothesis $L_0$ is $L$-stable, there is a neighborhood $L_0\subset V \subset W$
such that $\mbox{Sat}_\fa(V) \subset W$. As above we may assume that
 $V_p:=V\cap \Sigma$ is connected and   that
$V=\mbox{Sat}_\fa(V_p)$. In particular, we have $\mbox{Sat}_\fa(V)=V$, i.e., $V$ is saturated.  Since $V\subset W=\mbox{Sat}_\fa(U)$ we have $V\cap \Sigma \subset  W \cap \Sigma=U$.

\begin{Claim}
The pseudo-orbits of the points $x\in V_p$ under the virtual holonomy group $G:=\mbox{Hol}^{\;\virt}(\fa, L_0, \Sigma, p)$ remain in $U$.

\end{Claim}
\begin{proof}
Indeed, given $x \in V_p$ the $G$-pseudo-orbit of $x$, $G(x)$ is contained in the saturation  $\mbox{Sat}_{\fa}(x)$ of the point $x$,  and therefore $G(x)\subset \mbox{Sat}_\fa(V_p)\cap \Sigma = V_p \subset U$.
\end{proof}
This shows that the virtual holonomy group $\mbox{Hol}^{\virt}(\fa,L_0,\Sigma,p)$ is $L$-stable. \end{proof}

According to Lemma~\ref{Lemma:Lstablegroup} the (virtual) holonomy group $\Hol(\fa, L_0)$ is finite or linearizable containing some irrational rotation (circle type).
This already proves the first part of the theorem.
Let us consider case (i) and $M$ compact. Then by the stability theorem
\cite{brunellastability} the foliation is compact with stable leaves.
Assume now that we are in case (ii). Then according to the main construction in \cite{C-LN-S}  there is a closed meromorphic one-form $\Omega$ defined in an invariant neighborhood $\mathcal W$ of $L_0$, with polar set $(\Omega)_\infty=L_0$ of order one and defining $\fa$ in $\mathcal W\setminus L_0$. If $M$ is compact and $M\setminus L_0$ is a Stein manifold then $K:=M\setminus \mathcal W$ is a compact subset of $M\setminus L_0$ and $\Omega$ is meromorphic in $M\setminus K=\mathcal W$, so that by
Levi extension theorem \cite{C-LN-S} this form extends as a meromorphic one-form $\ov{ \Omega}$ to $M$. Because a Stein manifold cannot contain compact analytic subsets of positive dimension, this implies that the poles of $\ov{\Omega}$ are already contained in $\mathcal W$. Thus, except for $L_0$,  the leaves of $\fa$ have trivial holonomy group. Theorem~\ref{Theorem:global} is now proved. \qed

\bibliographystyle{amsalpha}

\end{document}